\documentclass[reqno]{amsart}
\usepackage{amsfonts,amssymb,amsmath,amsthm,array, graphicx, color, tikz,tikz-cd,soul,phonetic,mathrsfs,multirow}
\usepackage{url}
\usepackage[normalem]{ulem}
\usepackage{enumerate}
\usepackage{booktabs}
\usepackage{caption}
\newtheorem{theorem}{Theorem}[section]
\newtheorem{lemma}[theorem]{Lemma}
\newtheorem{proposition}[theorem]{Proposition}
\newtheorem{corollary}[theorem]{Corollary}
\theoremstyle{definition}
\newtheorem{definition}[theorem]{Definition}
\newtheorem{rem}[theorem]{Remark}
\newtheorem{example}[theorem]{Example}
\numberwithin{equation}{section}

\newcommand{\bte}{\begin{theorem}\quad  }
\newcommand{\ete}{\end{theorem} }
\newcommand{\bpr}{\begin{proposition}\quad  }
\newcommand{\epr}{\end{proposition} }
\newcommand{\ble}{\begin{lemma}\quad }
\newcommand{\ele}{\end{lemma}}
\newcommand{\bco}{\begin{corollary}\quad }
\newcommand{\eco}{\end{corollary} }
\newcommand{\bex}{\begin{example}\quad \rm }
\newcommand{\eex}{\end{example} }
\newcommand{\bde}{\begin{defi}\quad \rm }
\newcommand{\ede}{\end{defi} }
\newcommand{\brm}{\begin{rem} \quad \rm}
\newcommand{\erm}{\end{rem} }
\newcommand{\bpf}{\begin{proof}[{\bf{Proof.\quad}}] \rm}
\newcommand{\epf}{ \end{proof}}
\newcommand{\bdm}{\begin{displaymath} }
\newcommand{\edm}{\end{displaymath} }
\newcommand{\be}{\begin{eqnarray*}}
\newcommand{\ee}{\end{eqnarray*}  }

\newcommand{\lb}{\label}

\makeatletter
\newcommand\cupop{\mathop{\operator@font \cup}\nolimits}
\makeatother
\numberwithin{equation}{section}

\begin{document}
\author[Samira Hosseinzadeh]{Samira Hosseinzadeh Alikhalaji}
\address{Science and Research Branch, Islamic  Azad University, Tehran, Iran}
\email{samira.hosseinzadehalikhalaji@sriau.ac.ir}
\author{Mojtaba Sedaghatjoo$^{^*}$}
\address{Department of Mathematics, College of Sciences, Persian Gulf University, Bushehr, Iran.}
\email{sedaghat@pgu.ac.ir}
\author{Mohammad Roueentan}
\address{College of Engneering, Lamerd Higher Education Center, Lamerd, Iran.}
\email{rooeintan@lamerdhec.ac.ir}
\thanks{$^{^*}$Corresponding author}
\title[]{the structure of regular right uniform semigroups}
\begin{abstract}
In this paper we investigate right uniform notion on some classes of semigroups. The main objective of this paper is realizing the structure of regular right uniform semigroups which can be applied as a cornerstone of characterizing regular right subdirectly irreducible semigroups.
\end{abstract}

\maketitle
\section{INTRODUCTION AND PRELIMINARIES}
Taking inspiration from uniform modules, investigation on uniform acts over semigroups was initiated by Feller and Gantos in the category of {\bf Act}$_0-S$ (\cite{feller}). In a recent work on uniform acts \cite{MM}, an introductory account on uniform notion as a generalisation of (subdirectly) irreducible acts, in the category of {\bf Act}$-S$ is presented. The uniformness of a semigroup $S$ as a right $S$-act over itself is the main subject of this paper which shall be called a right uniform semigroup. The paper is organized in three sections. In the first section we present preliminaries and terminologies needed in the sequel. Section 2 contains general results on right uniform semigroups, mainely, we realize conditions under which right uniformness is transformed from a semigroup $S$ to $S^1$($S^0$) and vice versa. Moreover, we prove that the set of idempotents of a right uniform semigroup $S$ is a left zero or a right zero subsemigroup, and hence any right uniform semigroup is an $E$-semigroup. Section 3 is devoted to investigating right uniformness on some classes of semigroups, in particular, we realize the structure of regular right uniform semigroups which can be constructive in characterisation of regular right subdirectly irreducible semigroups. Besides we realize the structure of  right uniform semigroups which are subclasses of regular semigroups. Ultimately, we summarize all characterized classes of right uniform semigroups in a table.

Throughout this paper, $S$ will denote an arbitrary semigroup which is not a singleton. To every semigroup $S$ we can associate the monoid $S^1$ with identity element $1$ adjoined if necessary: $$S^1=\begin{cases} S & \text{if} ~S~ \text{has an identity element,} \\ S  \cup\{1\} & \text{otherwise}. \end{cases}$$
In a similar fashion to every semigroup $S$ we can associate the semigroup $S^0$ with the zero element $0$ adjoined if necessary. If $G$ is a (right, left) group, then $G^0$ will be called a (right, left) 0-group.

Recall that a right $S$-act $A_S$ (or $A$, if there is no danger of ambiguity) is a nonempty set together with a function $ \mu: A\times S \longrightarrow A$, called the action of $S$ on $A$ such that $a(st)=(as)t$ for each  $a\in A$ and  $s,t\in S$ ( denoting $\mu(a,s)$ by $\,as$). Hereby, any semigroup $S$ can be considered as a right $S$-act over itself with the natural right action, denoted by $S_S$. For two acts $A$ and $B$ over a semigroup $S$, a mapping $f: A\longrightarrow B$ is called a homomorphism of right $S$-acts or just an $S$-homomorphism if $f(as)=f(a)s$ for all $a\in A,\, s\in S$.  An element $\theta$ of an act  $A$ is said to be a zero element if $\theta s=\theta$ for all $s\in S$. The set of zero elements of an act $A$ is denoted by $Z(A)$. Moreover, the one-element act, denoted by $\Theta=\{\theta\}$, is called the zero act. An equivalence relation  $\rho$ on a right $S$-act $A$ is called a congruence on $A$ if  $a~\rho ~a'$ implies $(as)~ \rho ~(a's)$ for every $a,a' \in A, s \in S$. A congruence on $A_S=S_S$ is called a right congruence. For an act $A$ the diagonal relation $\{(a,a)\,|\,a\in A\}$ on $A$ is a congruence on $A$ denoted by $\Delta_A$. 


Also if $B$ is a subact of $A$, then the congruence $(B\times B)\cup \Delta_A$  on $A$ is denoted by $\rho_B$ and is called the Rees congruence by the subact $B$.  A subact $B$ of an act $A$ is called large in $A$ (or $A$ is called an essential extension of $B$) and is denoted by $B \subseteq' A$, if any $S$-homomorphism $g: A \longrightarrow C$ such that $g|_B$ is a monomorphism is itself a monomorphism. One may routinely observe that a subact $B$ of an act $A$ is large in $A$ if and only if, for every (principal) nondiagonal congruence $\rho \in$ Con$(A)$, $\rho_B \cap \rho \neq \Delta_A$. For a thorough account on the preliminaries, the reader is referred to \cite{kilp,MM}.

\section{Uniform semigroups}

In this section we present new observations on right uniform semigroups which are needed in the sequel.
\begin{definition} For a semigroup $S$, a right $S$-act $A$ is called $uniform$ if every nonzero subact is large in $A$. Also a semigroup $S$ is called right (left) uniform if the right (left) $S$-act $S_S$ ($_SS$) is uniform.
\end{definition}
To simplify by a uniform semigroup, we mean a right uniform semigroup. In what follows we state two required results from \cite{MM} which are employed in the next arguments.
\bco \lb{co5}
\cite[Proposition 2.5]{MM} Let $A$ be a uniform act over a semigroup $S$. Then $|Z(A)|\leq 2$.
\eco
\bco \lb{co1}
If $S$ is a uniform semigroup and $xy=y$ for $x,y\in S$, then $x$ is a left identity or $y$ is a left zero.
\eco
Note that the above condition is not sufficient for a semigroup $S$ to be uniform, for instance, any left zero semigroup with more than two elements satisfies the condition while it is not uniform. In the next proposition we prove that for commutative chain semigroups (ideals ordered under inclusion form a chain) the condition is sufficient.
 \begin{proposition} \label{pr2}
Suppose that $S$ is a commutative chain semigroup. Then $S$ is uniform if and only if for $x,y\in S$, $xy=y$ implies that $x$ is the identity element or $y$ is the zero element.
\end{proposition}
\begin{proof}
\textbf{Necessity} This is a straightforward consequence of Corollary \ref{co1}.

\textbf{Sufficiency} Since $S$ is commutative, $S_S$ has at most one zero element and regarding \cite[Lemma 2.18]{MM}, we show that every nonzero principal ideal of $S$ is large. Let $f:S\longrightarrow A$ be a homomorphism and for a nonzero element $x\in S, f_{|xS}$ is a monomorphism. If $z,y \in S$ and $f(z)=f(y)$ , without loss of generality suppose that $z=yt$ for some $t\in S$ and so $f(yt)=f(y)$. If $yS^1\subseteq xS^1$, since $y,z\in xS$, $z=y$. Otherwise, $xS^1\subseteq yS^1$ and for some $h\in S, x=yh$. Consequently $f(yt)h=f(y)h$ and using commutativity of $S, f(xt)=f(yth)=f(yh)=f(x)$. Accordingly $xt=x$ and since $x\neq 0,$ by Corollary \ref{co1}, $t=1$. Therefore $z=yt=y$ as desired.
\end{proof}
Our next aim is to prove that uniformness is transformed from $S^1(S^0)$ to $S$. We should mention that for a semigroup $S$, any right $S$-act or any $S$-homomorphism can be naturally considered as a $S^{1}$-act or $S^1$-homomorphism and vice versa. This observation is applied in the next two results implicitly.
\ble \label{pr1}
For a semigroup $S$, if ${S^{1} (S^{0} )}$ is uniform, then $S$ is uniform.
\ele
\begin{proof}
If $S^{1}$ is uniform, then $S$ is a large right ideal of $S^{1}$ and by \cite[Lemma 2.11]{MM}, $S$ is uniform as a $S^1$-act. Note that any nonzero right ideal of $S$ is large in $S$ as a right $S^1$-act and consequently as a ritht $S$-act. Therefore $S$ is right uniform.

Now suppose that $S$ is a semigroup with no zero element and $S^{0}$ is uniform. Let $f:S\longrightarrow A$ be a $S$-homomorphism such that $f_{|I}:I \longrightarrow A$ is a monomorphism where $I$ is a right ideal of $S$. Clearly, $A\sqcup \Theta$ is a $S^0$-act by $a0=\theta 0=\theta$ for any $a\in A$, and the mapping $f ^{0}:{S^{0}\longrightarrow {A\sqcup \Theta}}$ defined by: \begin{equation*} f^0(s)=
\begin{cases}
f(s)\ ; &\text{$s\in S$}\\
\theta\ ; &\text{s=0}
\end{cases}
\end{equation*}
is a $S^{0}$-homomorphism for which $f^{0}_{|I^{0}}$ is an $S^{0}$-monomorphism. The uniformness of $S^0$ implies that $f^0$ is a $S^0$-monomorphism and consequently $f$ is a $S$-monomorphism.
\end{proof}
The next example shows that for a semigroup $S$, uniformness is not transformed from $S$ to $S^1$ and $S^0$ generally.
\bex \lb{ex1} Let $S$ be a left zero semigroup with two elements. Clearly $S$ is uniform but $S^0$ has three left zeros and hence is not uniform by Corollary \ref{co5}. Now, let $S=\{e_1,e_2\}$ be a right zero semigroup. Clearly $S$ is uniform. Take the homomorphism $f:S^1 \longrightarrow S^1$ given by $1\mapsto e_1$ which is not a monomorphism. Since $S$ is a nonzero right ideal of $S^1$ and $f_{|S}$ is a monomorphism, $S$ is not large in $S^1$ and hence $S^1$ is not uniform.
\eex
Regarding the foregoing example our next aim is investigating conditions under which uniformness of $S$ implies uniformness of $S^1$ or $S^0$.
\bte \lb{th1}
Let $S$ be a semigroup with no identity element.  $S^{1}$ is uniform if and only if  $S$ is uniform and has no left identity element.
\ete
\begin{proof}
\textbf{Necessity} By Lemma \ref{pr1}, $S$ is uniform and so we need only to show that $S$ has no left identity. Suppose by way of contradiction that $e$ is a left identity of $S$. It can be routinely checked that $\{(1,e),(e,1)\}\cup \Delta_{S^1}$ is the principal right congruence $\rho(1,e)$ on $S^1$. Now $\rho(1,e)\cap \rho_{S}=\Delta _{S^1}$ which proves that $S$ is not a large right ideal of $S^1$, a contradiction.

\textbf{Sufficiency} Suppose that $f:S^{1}\longrightarrow A$ is an $S^{1}$-homomorphism for an $S^{1}$-act $A$ and for a nonzero and proper right ideal $I$ of $S^{1}, f_{|I}$ is a monomorphism. Then the mapping $f_{|S}:S \longrightarrow A$ can be considered as an $S$-homomorphism and since $I$ is also a nonzero right ideal of $S$, uniformness of $S$ implies that  $f_{|S}$ is a monomorphism. If $ f(1)=f(s)$, for some $s\in S$, then for any $t\in S, f(t)=f(1)t=f(s)t=f(st)$ that implies that $st=t$ for any $t\in S$ and so $s$ is a left identity of $S$, a contradiction. Therefore $f$ is a monomorphism and we are done.
\end{proof}
The next theorem is a counterpart version of the above theorem for $S^0$.

\bte \label{th2}
Let $S$ be a semigroup with no zero element. Then $S^{0}$ is uniform if and only if $S$ is uniform and has no left zero element.
\ete
\begin{proof}
\textbf{Necessity} By Lemma \ref{pr1}, $S$ is uniform and we just need to prove that $S$ has no left zero. By way of contradiction, suppose that $\theta$ is a left zero element of $S$. The mapping $f:S^0\longrightarrow S^0$  defined by
  \begin{equation*} f(s)=
\begin{cases}
\theta\  &\text{if~} s\in S\\
0\  &\text{if~} s=0,
\end{cases}
\end{equation*}
is an $S^0$-homomorphism. Since $I=\lbrace {\theta,0 }\rbrace $ is a right ideal of $S^0$ and $f _{|I}$ is a $S^0$-monomorphism, uniformness of $S^0$ implies that $f$ is a monomorphism and consequently, $S$ should be a singleton, a contradiction.

\textbf{Sufficiency} Suppose that $f:S^{0}\rightarrow A$ is a $S^{0}$-homomorphism for a $S^{0}$-act $A$ and for a nonzero right ideal $I$ of $S^{0}, f_{|I}$ is a monomorphism. Therefore $A$ can be considered naturally as a right $S$-act and $f_{|S}$ can be considered a $S$-homomorphism. Since $S$ has no zero element, $I\setminus {\lbrace 0\rbrace}$ is a nonzero right ideal of $S$ and $f_{|I\setminus{\lbrace 0\rbrace}}$ is an $S$-monomorphism. The uniformness of $S$ implies that $f_{|S}$ is also a $S$-monomorphism. Now, if for an element $z\in S$, $f(z)=f(0)$ then for any $s\in S,f(zs)=f(0s)= f(0)=f(z)$ and so $zs=z$ for any $s\in S$. Thus $z$ is a left zero element of $S$, a contradiction.  Therefore, $f:S^{0}\rightarrow A$ is a monomorphism.
\end{proof}
In the next proposition the structure of the set of idempotent elements $E(S)$, for any uniform semigroup $S$ is realized.

\bpr \lb{pr4}Let $S$ be a uniform semigroup such that $E(S)\neq \emptyset$. Then the structure of $E(S)$ is realized as follows:
\begin{enumerate} [{\rm i)}]
\item $E(S)=L$ or $L^1$ where $L$ is a two elements left zero semigroup,
\item $E(S)=R$ or $R^0$ where $R$ is a right zero semigroup.
\end{enumerate}
\epr
\bpf Due to the number of left zero elements in a uniform semigroup (Corollary \ref{co5}) two cases may occur.

{\bf Case 1:} $S$ has two left zero elements. By virtue of \cite[Coeollary 2.16]{MM} and \cite[Theorem 2.9]{Ran}, $E(S)=L$ or $L^1$ where $L$ consists of two left zero elements.

{\bf Case 2:} $S$ has one left zero element which is the zero element of $S$ or $S$ has no zero element. Regarding Corollary \ref{co1} any nonzero idempotent is a left identity element. Since the set of left identity elements in a semigroup forms a right zero subsemigroup, the set of nonzero idempotents form a right zero subsemigroup of $S$. Therefore $E(S)=R^0$ or $R$ where $R$ is a right zero semigroup.

\epf
Following \cite{almeida}, a semigroup $S$ is called an $E$-semigroup if $E(S)$ (the set of idempotent elements) is a subsemigroup. The next corollary is an immediate result of the above proposition.
\bco \lb{co9}
Any uniform semigroup is an $E$-semigroup.
\eco
\section{Uniform notion on some classes of semigroups}
In this section we investigate uniform notion on some classes namely, left simple, left 0-simple, regular, strongly right noetherian, completely simple and completely 0-simple semigroups.

It is clear that any right simple semigroup is uniform. In the next result we prove that any left simple uniform semigroup with more than two elements is a group.
\bco \lb{co8}
Let $S$ be a left simple semigroup. Then $S$ is uniform if and only if $S$ is a left zero semigroup with two elements or $S$ is a group.
\eco

\begin{proof}
The sufficiency part is clear and we just need to prove the necessity. Note that in any left simple semigroup, any left identity is an identity. Two cases may occur.
\\ {\bf Case 1:} $S$ contains a left zero element namely $z$. Then $Sz=S$ implies that $S$ is a left zero semigroup and by Corollary \ref{co5}, $|S|= 2$.\\ {\bf Case 2:} $S$ has no left zero element. For $x\in S$, the equality $Sx=S$ implies that $tx=x$ for some $t\in S$. By assumption and by Corollary \ref{co1}, $t$ is a left identity which is the identity element of $S$. Then $S$ is a left simple semigroup with an identity element and hence is a group.
\end{proof}
We recall from \cite{cliff1} that a semigroup $S$ with a zero element $0$ is called {\it left(right) 0-simple} if $S^2\neq 0$, and $0$ is the only proper left(right) ideal of $S$. In the next corollary we prove that left 0-simple uniform semigroups are exactly 0-groups.
\bco \lb{co7}
Let $S$ be a left  0-simple semigroup. Then $S$ is uniform if and only if $S$ is a 0-group.
\eco
\begin{proof}
Since the sufficiency part is clear, we prove the necessity part. Note that for any element $y \in S, Sy=S$ implies that $xy=y$ for some $x\in S$. Thus if $y$ is nonzero, since $S$ contains a zero element, $y$ is not a left zero and consequently by Corollary \ref{co1}, $x$ is a left identity.  Since $S=Sx$, for any nonzero element  $z\in S, z=tx$, for some $t\in S$. Thus $zx=tx^{2}=tx=z$ and so $x$ is also a right identity. Hence $S$ has an identity element and by \cite[Proposition 3.13]{MM}, $S$ is the disjoint union of its maximum subgroup $G$ and a two sided ideal $I$. Since $S$ is left 0-simple, $I=\{0\}$ and we are done.
\end{proof}

Following \cite{PG,MM}, a semigroup $S$ is called right(left) irreducible if the diagonal relation is meet irreducible in the class of right(left) congruences on $S$. Then  regarding the arguments at the end of section 1, we have the implications

\centerline {right uniform $\Longrightarrow$  right irreducible $\Longrightarrow$ right subdirectly irreducible }
for semigroups.

 In what follows we are going to characterize uniform semigroups with the ascending chain condition on right congruences which yields another approach on characterizing (subdirectly) irreducible finitely generated commutative semigroups investigated in \cite{PG} and leads to characterizing some other classes of semigroups, for instance finite irreducible semigroups. First we need to present some ingredients.

 Recall from \cite{MM} that a semigroup $S$ is called strongly right noetherian if it satisfies the ascending chain condition for right congruences. Then by \cite[R$\acute{\text{e}}$dei's Theorem]{cliff2}, finitely generated free commutative semigroups and consequently finitely generated commutative semigroup as quotients of such semigroups are strongly (right) noetherian. Also recall from \cite{Ch} that for an element $a$ of a right $S$-act $A$, the right congruence annihilator of $a$ is defined by ann$(a):=\lbrace(s,t)\in S\times S|~ as=at\rbrace=$ ker$\lambda_a$ where $\lambda_a:S_S \longrightarrow A$ is defined by $\lambda_{a}(s)=as$ for every $s\in S$. For any semigroup $S$ an element $s\in S$ shall be called {\it left nilpotent} if $s^n$ is a left zero element for some natural $n$. It is clear that a left and right nilpotent element is nilpotent. A semigroup $S$ is called {\it left nil} if all elements of $S$ are left nilpotent. Besides, we recall from \cite{PG} that a subelementary semigroup is a commutative semigroup $S$ which is the disjoint union $S=N\cup C$ of an ideal $N$ which is a nilsemigroup, and a subsemigroup $C(\neq \emptyset)$ every element of which is cancellative in $S$. Accordingly, we call a semigroup $S$ left subelementary, if it is the disjoint union $S=L\cup C$ of a left ideal $L$ which is a left nil semigroup, and a subsemigroup $C(\neq \emptyset)$ every element of which is left cancellable in $S$.

In \cite[Proposition 2.2]{PG}, it is proved that any element in a finitely generated commutative irreducible semigroup is cancellative or nilpotent, consequently, from a structural point of view, any finitely generated commutative irreducible semigroup is cancellative, nil or subelementary. So as a generalization of this result, in the next proposition we realize the structure of a strongly right noetherian uniform semigroup. Note that the proof is the same for \cite[Proposition 2.2]{PG}, but for more clarification we present the proof.
 \begin{proposition} \label{pr2.6}
 Suppose that $S$ is a strongly right noetherian uniform semigroup. Then $S$ is left cancellative, left nil or left subelementary.
   \end{proposition}

 \begin{proof}
For an element $c\in S$ we have the ascending chain of right congruences $\text{ann}(c)\subseteq\text{ann}(c^2) \subseteq \cdots$ . Our assumption necessitates that for some $n\in \mathbb{N}$, $\text{ann}(c^n)=\text{ann}(c^{n+1})=\cdots=\text{ann}(c^{2n})$. For the right ideal $I=c^n S$ we show that $\rho_I \cap \text{ann}(c^n)=\Delta_S$. Let $(x,y)\in \rho_I \cap\text{ann}(c^n) $ and $x\neq y$. Then $x,y\in I $ and $c^n x=c^n y$. Hence for some $u,v \in S, x=c^n u$ and $y=c^n v$ which implies that $c^{2n} u=c^n x=c^n y=c^{2n} v$. Therefor $(u,v)\in \text{ann}(c^{2n})=\text{ann}(c^n)$ and consequently $ x=c^n u=c^n v=y$, a contradiction. Now by uniformness of $S$, ann$(c^n)=\Delta_S$ or $I$ is a zero subact of $S_S$. In the first case $\text{ann}(c) \subseteq$ $\text{ann}(c^n)$ implies that $\text{ann}(c)=\Delta_S$ or equivalently $c$ is left cancellable. In the second case $c^{n+1}$ is a left zero element or equivalently $c$ is left nilpotent. To complete the proof we need to prove that the set of left nil elements in $S$ is a left ideal. Suppose that $c\in S$ is a left nilpotent element. If $c$ is a left zero element then $sc$ is a left zero element for any $s\in S$ and hence $sc$ is left nilpotent. Otherwise there is a natural $n>1$ such that $c^n$ is a left zero element but $c^{n-1}$ is not. Now, for an arbitrary element $s\in S$, $(sc)c^{n-1}=(sc)c^n$ which shows that $sc$ is not left cancellable and hence is  left nilpotent.
\end{proof}

It should be mentioned that in \cite{MM} it is proved that any strongly right noetherian uniform semigroup with no left zero element is left cancellative, which the above proposition generalizes this result.  As a result of the above proposition and \cite[Corollary 3.11]{MM}, the next corollary is deduced.

\bco \lb{co2}
 Any finite uniform semigroup is right group, left nil or left subelementary.
 \eco
 The rest of paper is allocated to characterizing regular right uniform semigroups. In \cite[Theorem 3.17]{MM} the structure of regular uniform monoids is identified. In the next proposition, regular uniform semigroups without left identity elements are characterized.
 \bpr \label{pr3}
 Suppose that $S$ is a regular semigroup which has no left identity. Then $S$ is uniform if and only if $S=\lbrace {\theta_1,\theta_2}\rbrace$ where $\theta_1,\theta_2$ are left zero elements.
 \epr
\begin{proof}
The sufficiency part is clear.

\textbf{Necessity} By virtue of Theorem \ref{th1}, $S^1$ is a regular uniform monoid and hence by \cite[Theorem 3.17]{MM} has one of the following structures:
\begin{enumerate} [{\rm i)}]
\item $S^1$ is  a group,
\item $S^1=G^0,$ where $G$ is a group,
\item $S^1=G\sqcup \lbrace{\theta_1, \theta_2}\rbrace$ where $G$ is group, ${\theta_1, \theta_2}$ are left zero elements and $s\theta_i=\theta_j$ for all $s\in G\setminus \lbrace{1}\rbrace ,1\leq i\neq j\leq 2$.
\end{enumerate}
Since in all cases $S$ is a proper right ideal of $S^1$, $S^1$ is not a group and hence the first case does not occur. In other cases, as the identity element of $S^1$ is the identity element of $G$, if $x\in S \cap G$ then $S$ contains the identity element, a contradiction. Since $|S| > 1$, $S=\{\theta_1,\theta_2\}$ and we are done.
\end{proof}
As a result of the above proposition we conclude that almost all regular uniform semigroups have left identities.
\bco \label{co3}
Any regular uniform semigroup is a left zero semigroup with two elements or contains a left identity.
\eco

In the next proposition we characterize regular uniform semigroups without identity element but possessing  a left identity.
\bpr \lb{pr2} Let $S$ be a regular semigroup with a left identity $e$ which is not a right identity. $S$ is uniform if and only if it is a right group or a right 0-group.
\epr

\bpf The sufficiency part is clear.

{\bf Necessity:} First we prove that $S$ has more than one left identity and as a result, $S$ has at most one left zero element.

Let $S$ be a regular uniform semigroup with no identity element and let $e$ be a left identity element. If $e$ is the only left identity in $S$, regularity of $S$ implies that for any non left zero element $s\in S$, $s=sf$ for an idempotent $f$. Since $s$ is not a left zero element, $f$ is not a left zero and regarding Corollary \ref{co1}, $e=f$. Thus $e$ is a right identity and consequently it is the identity element, a contradiction. So there is a left identity namely $f\neq e$ in $S$. Therefore, $\rho(e,f)=\{(e,f),(f,e)\}\cup \Delta_S$  is a nondiagonal right congruence on $S$. If $S$ has two left zeros namely $\theta_1,\theta_2$, then $\rho(e,f)\cap \rho_{\{\theta_1,\theta_2\}}=\Delta_S$ which is a contradiction. Thus $S$ has at most one left zero element. Now the uniformness of $S$ implies that $e,f\in sS$  for any nonzero element $s\in S$. Substituting $e,f$ by any pair of distinct nonzero idempotents, we conclude that $E(S)\subseteq sS$ for any nonzero element $s\in S$. Since $S$ is regular, $S=E(S)S$ and hence for any nonzero element $s\in S,~S=E(S)S\subseteq sS$. Therefore, if $S$ has no zero element it is a right simple semigroup containing an idempotent and hence $S$ is a right group by \cite[Theorem 1.27]{cliff1}. Otherwise, if $S$ has the zero element 0, and $ab=0$, for nonzero elements $a,b\in S$, then $I=\{s\in S \,|\,as=0\}$ is a nonzero right ideal of $S$ and hence $I=S$. Therefore $aS=0$. Now, since $S$ is regular, $a\in aS=0$, a contradiction. Therefore, $S\backslash \{0\}$ is a regular right simple semigroup and hence $S$ is a right 0-group.
\epf
As a result of \cite[Theorem 3.17]{MM}, Propositions \ref{pr3}, \ref{pr2} and corollary \ref{co3} we can characterize regular uniform semigroups.
\bte \lb{th3} Let $S$ be a regular semigroup. $S$ is uniform if and only if $S$ has one of the following structures:
\begin{enumerate} [{\rm i)}]
\item $S=G$ or $G^0$, where $G$ is a group,
\item $S=G\sqcup \lbrace{\theta_1, \theta_2}\rbrace$ where $G$ is a group, ${\theta_1, \theta_2}$ are left zero elements and $g\theta_i=\theta_j$ for all $g\in G\setminus \lbrace{1}\rbrace ,1\leq i\neq j\leq 2$.
\item $S=\lbrace{\theta_1, \theta_2}\rbrace$ where  $\theta_1, \theta_2$ are left zero elements.
\item $S=G$ or $G^0$, where $G$ is a right group,
\end{enumerate}
where the first two structures are due to regular uniform monoids and the last two ones are due to regular uniform semigroups without identity element.
\ete

Since uniform semigroup are generalisations of right (subdirectly) irreducible semigroups, characterization of regular uniform semigroups in Theorem \ref{th3} can lead to characterization of regular right subdirectly irreducible semigroups. But it seems  we need a more sophisticated description of right congruences of regular semigroups.

Regarding the fact that completely simple and completely 0-simple semigroups involve classes namely right groups, left groups, right zero semigroups, left zero semigroups and rectangular bands in the next propositions we investigate uniform notion for such semigroups. We recall the Rees Theorem for completely (0-)simple semigroups, which exhibits a decomposition of such semigroups in terms of Rees matrix semigroups over (0-)groups with the regular sandwich matrixes (see \cite{How}).

As a result of Theorem \ref{th3} we have the following characterization of uniform completely (0-)simple semigroups.
\bco  \lb{co4} Let $S=\mathcal{M}^{0}[G;I,\Lambda;P]$ be a completely 0-simple semigroup. $S$ is right uniform if and only if $I$ is a singleton or equivalently, $S$ is a right 0-group.
\eco
\bpf Since $S$ has no left zero element other than the zero element, regarding the structure of regular uniform semigroups in Theorem \ref{th3}, $S$ is a 0-group or right 0-group which is equivalent to $I$ being a singleton.
\epf

\bco \lb{co6} Let $S=\mathcal{M}[G;I,\Lambda;P]$ be a completely simple semigroup. $S$ is right uniform if and only if $|I|=2, |\Lambda |=1$ and $G$ is a trivial group or $|I|=1$, equivalently, $S$ is a left zero semigroup with two elements or $S$ is a right group.
\eco
\bpf Since $S$ is simple and has no zero element, regarding Theorem \ref{th3}, $S$ is a right group or a left zero semigroup with two elements.
\epf
Note that Theorem \ref{th3} yields characterization of some known subclasses of regular semigroups which are uniform for instance {\it Clifford, completely regular, inverse, left and right inverse, band and orthodox} semigroups. Regarding Theorem \ref{th3}, Proposition \ref{pr2.6}, Corollaries  \ref{co9}, \ref{co8}, \ref{co7}, \ref{co2}, \ref{co4}, and  \ref{co6} in the next table we summaries characterizations of obtained classes of uniform semigroup in this paper.

\begin{table}[h]
\caption{\bf Characterizations of some classes of uniform semigroups}
\hspace*{-3cm}
\begin{tabular}{m{4cm}|p{8cm}|c}
     
   \centering    {\bf Class of uniform semigroup }                      & \centering {\bf Structure}                                                                                  &  {\bf Observation}                                   \\ \toprule
 \centering $S$ is regular or Orthodox or completely regular & $S$ has one of the structures in the statement of Theorem \ref{th3}                        & Theorem \ref{th3}~ and Corollary \ref{co9}                            \\ \hline
       \centering $S$ is right inverse                             & \centering S=$G$ or $G^0$ where $G$ is a group or a right group                                       & \multirow{2}{*}{\shortstack {Theorem \ref{th3} and \cite[Lemma 1.3.39]{kilp}  \\on the structure of right inverse \\ and left inverse semigroups}} \\ \cline{1-2}
      \centering  $S$ is left inverse                              & \centering $S$ has one of the structures ${\rm i),ii)}$ in the statement of Theorem \ref{th3}     &  \\ \hline
      \centering  $S$ is inverse                                   & \centering S=$G$ or $G^0$ where $G$ is a group                                                        & Theorem \ref{th3}                             \\ \hline
      \centering  $S$ is Clifford                                  &\centering S=$G$ or $G^0$ where $G$ is a group                                                        &  Theorem \ref{th3}                             \\ \hline
      \centering  $S$ is completely simple                         &\centering $S$ is a left zero semigroup with two elements or $S$ is a right group.                    & Corollary \ref{co6}                           \\ \hline
      \centering  $S$ is completely 0-simple                       &\centering $S$ is a right 0-group                                                                     & Corollary \ref{co4}                           \\ \hline
       \centering $S$ is band                                      & \centering $S=\{0,1\}$ or $S$ is a right zero semigroup or $S$ is a two elements left zero semigroup  & Theorem \ref{th3}                             \\ \hline
\centering $S$ is left simple & \centering $S$ is a left zero semigroup with two elements or $S$ is a group & Corollary \ref{co8} \\ \hline
\centering $S$ is left 0-simple & \centering $S$ is a 0-group & Corollary \ref{co7} \\ \hline
\centering $S$ is strongly right noetherian &  \centering $S$ is left cancellative, left nil or left subelementary & Proposition \ref{pr2.6} \\ \hline
\centering $S$ is finite & \centering  $S$ is right group, left nil or left subelementary & Corollary \ref{co2}
    \end{tabular}

\end{table}

\end{document}